\documentclass[11pt,reqno]{amsart} % please use amsart at 11pt

\usepackage{amssymb,latexsym}
\usepackage[draft=true]{hyperref}
\usepackage{cite} % to get Refs. [1,2,3] typeset as [1--3] automatically

\usepackage[height=190mm,width=130mm]{geometry} % this is the journal
\theoremstyle{plain}
\newtheorem{theorem}{Theorem}
\newtheorem{lemma}{Lemma}
\newtheorem{corollary}{Corollary}
\newtheorem{proposition}{Proposition}

\theoremstyle{definition}

\theoremstyle{remark}

\numberwithin{equation}{section} % to get equations numbered
\begin{document}
\title[Measures of Noncompactness in $\bar{N}(p,q)$ Summable Sequence Spaces]{Measures of Noncompactness in $\bar{N}(p,q)$ Summable Sequence Spaces} % please provide
                                % an abbreviated title 

\author{Ishfaq Ahmad Malik}
\address{National Institute of Technology, Srinagar\\ Department of Mathematics\\ Srinagar\\ Jammu and Kashmir\\ India}

\email{ishfaq$\_$2phd15@nitsri.net}

\author{Tanweer Jalal}

\address{National Institute of Technology, Srinagar\\ Department of Mathematics\\ Srinagar, 190006\\ Jammu and Kashmir\\ India}
\email{tjalal@nitsri.net}

\begin{abstract}
In this paper, we first define the $\bar{N}(p,q)$ summable sequence spaces and obtain some basic results related to these spaces. The necessary and sufficient conditions for an infinite matrix $A$ to map these spaces into the spaces $c,~c_0,$ and $\ell_{\infty}$ is obtained and Hausdorff measure of noncompactness is then used to obtain the necessary and sufficient conditions for the compactness of linear operators defined on these spaces.
\end{abstract}

%\dedicatory{This paper is dedicated to Professor X on his 125th birthday.}



\maketitle
\section{Introduction}
\parindent=0mm\vspace{0.00in}
We write $\omega$ for the set of all complex sequences $x =(x_k)_{k=0}^{\infty}$ and $\phi$, $c$, $c_{0}$ and  $\ell_{\infty}$ 
for the sets of all finite sequences, convergent sequences,  sequences convergent to zero, and bounded sequences respectively.   
By $e$ we denote the sequence of 1's, $e=(1 ,1,1,\ldots)$ and by $e^{(n)}$ the sequence with 1 as only nonzero term at the $n$th place for each $n\in \mathbb{N}$, where $\mathbb{N}=\{0,1,2,\ldots\}$. Further by $cs$ and $\ell_1$ we denote the convergent and absolutely convergent series respectively. If $x=(x_k)_{k=0}^{\infty}\in w$ then $x^{[m]}=\sum_{k=0}^{m} x_ke^{(k)} $ denotes the $m-$th section of $x$.\\
A sequence space $X$ is a linear subspace of $\omega$, such a space is called a BK space if it is a Banach space with continuous coordinates\\
$P_n:X\rightarrow\mathbb{C}\hspace{0.2cm}(n=0,1,2,\ldots)$ where $$P_n(x)=x_n,~x=(x_k)_{k=0}^{\infty}\in X.$$ 
The BK space $X$ is said to have AK if every $x=(x_k)_{k=0}^{\infty}\in X$ has a unique representation $x=\sum_{k=0}^{\infty} x_ke^{(k)}$ \cite[Definition 1.18]{malkowsky1}. The spaces $c_{0}$, $c$ and $\ell_\infty$ are BK spaces with respect to the norm 
$$\|x\|_\infty=\sup_{k}\{|x_k|: k\in \mathbb{N}\}.$$ 
The $\beta-$dual of a subset $X$ of $\omega$ is defined by 
$$X^{\beta}=\left\{a\in w:ax=(a_kx_k)\in cs~~\text{for all}~x=(x_k)\in X \right\}$$
If $X$ and $Y$ are Banach Spaces, then by $\mathcal{B}(X,Y)$ we denote the set of all bounded (continuous) linear operators $L:X\rightarrow Y$ , which is itself a Banach space with the operator norm $\|L\|=\sup_{x}\left \{\|L(x)\|_Y:\|x\|=1 \right\}$ for all $L\in \mathcal{B}(X,Y).$ The linear operator $L:X\rightarrow Y$ is said to be compact if its domain is all of $X$ and for every bounded sequence $(x_n)\in X$, the sequence $\left(L(x_n)\right)$ has a subsequence which converges in $Y$. The operator $L\in \mathcal{B}(X,Y)$ is said to be of finite rank if $\dim R(L)<\infty$, where $R(L)$ denotes the range space of $L$. A finite rank operator is clearly compact \cite[Chapter 2]{diagana2013almost}.\\
In this paper, we first define  $\bar{N}(p,q)$ summable sequence spaces as the matrix domains $X_T$ of arbitrary triangle $\bar{N}_{p}^{q}$ and obtain some basic results related to these spaces. We then find out the necessary and sufficient condition for matrix transformations to map these spaces into $c_0$, $c$ and $\ell_{\infty}$. Finally we characterize the classes of compact matrix operators from these spaces into $c_0$, $c$ and $\ell_{\infty}$.

\section{Matrix Domains}

\parindent=0mm\vspace{0.00in}
Given any infinite matrix $A=(a_{nk})_{n,k=0}^{\infty}$ of complex numbers, we write $A_n$ for the sequence in the $n$th row of $A$, $A_n=( a_{nk})_{k=0}^{\infty}$ . The $A-$transform of any $x=(x_k)\in \omega$ is given by $Ax=\left(A_n(x)\right)_{k=0}^{\infty}$, where 
$$ A_n(x)=\sum_{k=0}^{\infty} a_{nk}x_k~~~~~~~~~n\in \mathbb{N}$$
the series on right must converge for each $n\in \mathbb{N}$.\\
If $X $ and $Y$ are subsets of $\omega$, we denote by $(X,Y)$, the class of all infinite matrices  that map $X$ into $Y$. So $A\in (X,Y)$ if and only if $A_n\in X^{\beta} ~,~n=0,1,2,\ldots$ and $Ax\in Y$ for all $x\in X$. The matrix domain of an infinite matrix $A$ in $X$ is defined by 
$$X_A=\left\{x\in \omega:Ax\in X\right\}$$
The idea of constructing a new sequence space by means of the matrix domain of a particular limitation method has been studied by several authors see \cite{jalal2012Diff, bocong2013matrix, jacob1977matrix, djolovic2008matrix, jalal2012newmatrix, jalal201Mat}\\
For any two sequences $x$ and $y$ in $\omega$ the product $xy$ is given by $xy=(x_ky_k)_{k=0}^{\infty}$ and for any subset $X$ of $\omega$
$$y^{-1}*X= \left\{ a\in \omega:ay\in X\right\}$$
We denote by $\mathfrak{U}$ the set of all sequences $u=(u_k)_{k=0}^{\infty}$ such that $u_k\not=0~~\forall ~k=0,1,2,\ldots$ and for any  $u\in \mathfrak{U}$, ${1\over u}=\left({1\over u_k} \right)_{k=0}^{\infty}$. 
\begin{theorem}\label{T11}
a) Let $X$ be a BK space with basis $(\alpha^{(k)})_{k=0}^{\infty}$, $u\in \mathfrak{U}$ and $\beta^{(k)}=(1/u)\alpha^{(k)}$, $k=0,1,\ldots$. Then $(\beta^{(k)})_{k=0}^{\infty}$ is a basis of $Y=u^{-1}*X$.\\
b) Let $(p_k)_{k=0}^{\infty}$ be a positive sequence , $u\in \mathfrak{U}$ a sequence such that 
$$|u_0|\leq |u_1|\leq \cdots~~~\text{ and }~|u_n|\rightarrow\infty~~(n\rightarrow\infty)  $$
and $T$ a triangle with 
$$ t_{nk}=\left\{\begin{matrix}
\frac{p_{n-k}}{u_n} & 0\leq k\leq n\\
0& k>n
\end{matrix}\right. ~~~~~~~~~~~n=0,1,2,\ldots$$
Then $(c_0)_T$ has AK.\\
c) Let $T$ be an arbitrary triangle and $B=|T|$. Then $(c_0)_{[B]}$ has AK if and only if $\lim_{n\rightarrow \infty} t_{nk}=0$ for all $k=0,1,2,\ldots$.   
\end{theorem}
\begin{proof}
a) (cf. \cite[Theorem 2]{Al-Em})\\
b) $(c_0)_T$ is a BK space by Theorem 4.3.12 in \cite{wilansky2000summability}, the norm $\|x\|_{(c_0)_T}$ on it is defined as \\
$$\|x\|_{(c_0)_T}=\sup_{n}\left|{1\over u_n}\sum_{k=0}^{n} p_{n-k} x_k\right| $$
Since $|u_n|\rightarrow\infty ~(n\rightarrow\infty)$ gives $\phi \subset (c_0)_T$. \\
Let $\epsilon>0$ and $x\in (c_0)_T$ then their exists integer $N>0$, such that $|T_n(x)|<{\epsilon\over 2}$ for all $n\geq N$.\\
Let $m>N$ then \\
\begin{align}
\|x-x^{[m]}\|_{(c_0)_T}=\sup_{n\geq m+1} \left|{1\over u_n} \sum_{k=m+1}^{n} p_{n-k}x_k\right|    \label{1}
\end{align}
Now \\
\begin{align*}
T_n(x)&={1\over u_n}\sum_{k=0}^{n} p_{n-k} x_k \\
T_m(x)&={1\over u_n}\sum_{k=0}^{m} p_{n-k} x_k \\ 
\Rightarrow T_n(x)+T_m(x)&={1\over u_n}\left[ 2(p_n x_0+\cdots+p_{n-m}x_m)+\sum_{k=m+1}^{n}p_{n-k}x_k\right]
\end{align*}
Then by (\ref{1}) we have 
\begin{align*}
\|x-x^{[m]}\|_{(c_0)_T}&\leq \sup_{n\geq m+1}(|T_n(x)|+|T_m(x)|)  \\
&<{\epsilon\over 2}+{\epsilon\over 2}\\
&=\epsilon
\end{align*}  
Hence $x=\sum_{k=0}^{\infty} x_k\beta^{(k)}$.\\
This representation is obviously unique.\\ 
c) Same as done in (cf. \cite[Theorem 2]{Al-Em})
\end{proof}
\section{$\bar{N}(p,q)$ Summable Sequence Spaces}
Let $(p_k)_{k=0}^{\infty}~,~(q_k)_{k=0}^{\infty}$ be positive sequences in $\mathfrak{U}$ and $(R_n)_{n=0}^{\infty}$ the sequence  with $R_n=\sum_{j=0}^{n} p_{n-j}q_{j} $. The $\bar{N}(p,q)$ transform of the sequence $(x_k)_{k=0}^{\infty}$ is the sequence $(t_n)_{n=0}^{\infty}$ defined as 
\begin{align*}
t_n={1\over R_n}\sum_{j=0}^{n}p_{n-j}q_j x_j
\end{align*}
The matrix $\bar{N}_{p}^{q}$ for this transformation is  
\begin{equation}\label{E11}
(\bar{N}_{p}^{q})_{nk}=\left\{\begin{matrix}
{p_{n-k}q_{k}\over R_{n}}& 0\leq k\leq n\\
0& k>n
\end{matrix}\right. 
\end{equation} 
We define the spaces $(\bar{N}_{p}^{q})_{0}$, $(\bar{N}_{p}^{q})$ and $(\bar{N}_{p}^{q})_{\infty}$ that are $\bar{N}(p,q)$ summable to zero, summable and bounded respectively  as  
\begin{align*}
&(\bar{N}_{p}^{q})_0&=(c_0)_{\bar{N}_{p}^{q}}&=\left\{x\in \omega: 
 \bar{N}_{p}^{q}x=\left({1\over R_n}\sum_{k=0}^{n}p_{n-k}q_k x_k \right)_{n=0}^{\infty} \in c_0 \right\}\\
 &(\bar{N}_{p}^{q})&=(c)_{\bar{N}_{p}^{q}}&=\left\{x\in \omega: 
 \bar{N}_{p}^{q}x =\left({1\over R_n}\sum_{k=0}^{n}p_{n-k}q_k x_k \right)_{n=0}^{\infty} \in c \right\} \\
 &(\bar{N}_{p}^{q})_{\infty}&=(\ell_{\infty})_{\bar{N}_{p}^{q}}&=\left\{x\in \omega: \bar{N}_{p}^{q}x=\left({1\over R_n}\sum_{k=0}^{n}p_{n-k}q_k x_k \right)_{n=0}^{\infty} \in \ell_{\infty} \right\}
\end{align*}  

\parindent=0mm\vspace{0.00in}
For any sequence $x=(x_k)_{k=0}^{\infty} $, define $\tau=\tau(x)$ as the sequence with $n$th term given by 
\begin{equation}\label{E12}
\tau_n=(\bar{N}_{p}^{q})_n(x)={1\over R_n}\sum_{k=0}^{n} p_{n-k}q_k x_k~~~~~~~~~~~~~~~~~~(n=0,1,2,\ldots)
\end{equation}
This sequence $\tau$ is called as \emph{weighted means of $x$}. 
\begin{theorem}\label{BKspace}
The spaces $(\bar{N}_{p}^{q})_0$, $ (\bar{N}_{p}^{q})$ and $(\bar{N}_{p}^{q})_{\infty}$ are BK spaces with respect to the norm $\| ~.~\|_{\bar{N}_{p}^{q}}$ given by
\begin{align*}
\| x\|_{\bar{N}_{p}^{q}}=\sup_{n}\left| {1\over R_n}\sum_{k=0}^{n}p_{n-k}q_k x_k \right| 
\end{align*}  
If $R_n\rightarrow \infty$ ($n\rightarrow \infty$), then $(\bar{N}_{p}^{q})_0$ has AK, and every sequence $x=(x_k)_{k=0}^{\infty}\in (\bar{N}_{p}^{q})$ has unique representation 
\begin{align}
x=le+\sum_{k=0}^{\infty}(x_k-l)e^{(k)}   \label{E13}
\end{align}
where $l\in \mathbb{C}$ is such that $x-le\in (\bar{N}_{p}^{q})_0$
\end{theorem}

\begin{proof}
The sets  $(\bar{N}_{p}^{q})_0$, $ (\bar{N}_{p}^{q})$ and $(\bar{N}_{p}^{q})_{\ell_{\infty}}$  are BK spaces   (\cite{wilansky2000summability}  {Theorem 4.3.12}),\\
Let us consider the matrix $T=(t_{nk})$ defined by 
$$t_{nk}=\left\{\begin{matrix}
\frac{p_{n-k}}{R_n} & 0\leq k\leq n\\
0& k>n
\end{matrix}\right. ~~~~~~~~~~~n=0,1,2,\ldots$$
Then $(\bar{N}_{p}^{q})_0=q^{-1}*(c_0)_T$ has AK by Theorem \ref{T11}.\\
Now if $x\in (\bar{N}_{p}^{q})$, then there exists a $l\in \mathbb{C}$ such that $x-le \in(\bar{N}_{p}^{q})_0$\\
Now $\tau(e)=(\tau_n)_{n=0}^{\infty}$  where 
\begin{align*}
\tau_n=(\bar{N}_{p}^{q})_n(e)&={1\over R_n}\sum_{k=0}^{\infty} p_{n-k}q_k e_k~~~~~~~~~(n=0,1,2,\ldots)\\
&={1\over R_n}\sum_{k=0}^{\infty} p_{n-k}q_k \hspace{1cm} \text{As } ~e_k=1~\forall~(k=0,1,2,\ldots)\\
&=1 
\end{align*}
Therefore $\tau(e)=e$ which implies the uniqueness of $l$.\\
Therefore \eqref{E13} follows from the fact that $(\bar{N}_{p}^{q})_{\infty}$ has AK.
\end{proof}

\parindent=5mm\vspace{0.00in} 
 Now $\bar{N}_{p}^{q}$ is a triangle, it has a unique inverse and the inverse is also a triangle \cite{Al-Em03}. Take $H_0^{(p)}={1\over p_0}$ and 
\begin{align}\label{H}
H_{n}^{(p)}={1\over p_{0}^{n+1}}\begin{vmatrix}
p_1&p_0&0&0&\ldots&0 \\
p_2&p_1&p_0&0&\ldots&0 \\
\vdots&\vdots&\vdots&\vdots&\ddots &\vdots\\
p_{n-1}&p_{n-2}&p_{n-3}&p_{n-4}&\ldots &p_0\\
p_{n}&p_{n-1}&p_{n-2}&p_{n-3}&\ldots &p_1
\end{vmatrix}
\end{align}
Then the inverse of matrix defined in  \eqref{E11} is the matrix $S=\left(s_{nk}\right)_{n,k=0}^{\infty}$ which is defined as see \cite{Diff15} in 
\begin{align}\label{E14}
s_{nk}=\left\{\begin{matrix}
(-1)^{n-k}\frac{H_{n-k}^{(p)}}{q_n}R_k&0\leq k\leq n \\ 0&k>n 
\end{matrix}\right.
\end{align}

\subsection{$\beta$ dual of $\bar{N}(p,q)$ Sequence Spaces}
In order to find the $\beta$ dual we need the following results \\
\begin{lemma}\label{L2} \cite{st1} If $A=(a_{nk})_{n,k=0}^{\infty}$, then 
$A\in (c_0,c)$ if and only if 
\begin{equation} \label{LE1}
\sup_{n}\sum_{k=0}^{\infty}\left|a_{nk} \right|<\infty
\end{equation}
\begin{equation}\label{LE2}
\lim_{n\rightarrow\infty} a_{nk}-\alpha_k=0\hspace{1cm} \text{for every }~ k.
\end{equation}
\end{lemma}
\begin{lemma} \label{L3} \cite{cooke2014infinite} If $A=(a_{nk})_{n,k=0}^{\infty}$, then 
$A\in (c,c)$ if and only if conditions \eqref{LE1}, \eqref{LE2} holds and 
\begin{equation} \label{LE3}
\lim_{n\rightarrow\infty} A_n=\lim_{n\rightarrow\infty} a_{nk}\hspace{1cm} \text{exists for all } k
\end{equation} 
\end{lemma}
\begin{lemma} \label{L4} \cite{cooke2014infinite} If $A=(a_{nk})_{n,k=0}^{\infty}$, then 
$A\in (\ell_{\infty},c)$ if and only if condition \eqref{LE2} holds and 
\begin{equation} \label{LE4}
\lim_{n\rightarrow\infty} \sum_{k=0}^{\infty}|a_{nk}|=\sum_{k=0}^{\infty} \left|\lim_{n\rightarrow\infty} a_{nk}\right| 
\end{equation} 
\end{lemma}\begin{theorem}\label{T2}
Let $\left(p_k\right)_{k=0}^{\infty}~,~\left(q_k\right)_{k=0}^{\infty}$ be positive sequences, $R_n=\sum_{j=0}^{n}p_{n-j} q_{j} $ and  $a=(a_k)\in \omega$ ,  we define a matrix $C=(c_{nk})_{n,k=0}^{\infty}$ as , 
 \begin{align}\label{E15}
 c_{nk}=\left\{\begin{matrix}
R_k \left[\sum_{j=k}^{n} \left(-1\right)^{j-k}\left({H_{j-k}^{(p)}\over q_{j}}a_j\right) \right] &0\leq k\leq n\\ 0& k>n \end{matrix}\right.
 \end{align}
and consider the sets 
\begin{align*}
c_1&=\left\{a\in \omega:\sup_{n}\sum_{k} |c_{nk}|<\infty\right\} &;
c_2&=\left\{a\in \omega:\lim_{n\rightarrow\infty} c_{nk} \text{ exists for each }~k\in \mathbb{N}\right\}\\
c_3&=\left\{a\in \omega:\lim_{n\rightarrow\infty}\sum_{k} |c_{nk}|=\sum_{k} \left|\lim_{n\rightarrow\infty} c_{nk}\right|\right\}&;
c_4&=\left\{a\in \omega:\lim_{n\rightarrow\infty} \sum_{k} c_{nk} \text{ exists }\right\}
\end{align*}
Then $\left[\left(\bar{N}_{p}^{q}\right)_0\right]^{\beta}=c_1\cap c_2$, 
$\left[\left(\bar{N}_{p}^{q}\right)\right]^{\beta}=c_1\cap c_2\cap c_4$ and 
$\left[\left(\bar{N}_{p}^{q}\right)_\infty\right]^{\beta}=c_2\cap c_3$. 
\end{theorem}  
\begin{proof}
We prove the result for $\left[\left(\bar{N}_{p}^{q}\right)_0\right]^{\beta}$.\\
Let  $x\in \left(\bar{N}_{p}^{q}\right)_0$ then there exists a $y$ such that $y= \bar{N}_{p}^{q} x$.\\
Hence 
\begin{align*}
\sum_{k=0}^{n}a_kx_k&=\sum_{k=0}^{n}a_k\left(\bar{N}_{p}^{q}\right)^{-1}y_k\\
&=\sum_{k=0}^{n}a_k \left[\sum_{j=0}^{k}(-1)^{k-j}R_j\left({H_{k-j}^{(p)}\over q_{k}}\right) y_j\right]\\
&=\sum_{k=0}^{n} R_{k}\left[\sum_{j=k}^{n} \left(-1\right)^{j-k}\left({H_{j-k}^{(p)}\over q_{j}}a_j\right) \right]y_k\\
&=(Cy)_n
\end{align*}
So $ax=(a_nx_n)\in cs$ whenever $x\in \left(\bar{N}_{p}^{q}\right)_0$ if and only if $Cy\in cs$ whenever $y\in c_0$.\\
Using Lemma \ref{L2} we get  $\left[\left(\bar{N}_{p}^{q}\right)_0\right]^{\beta}=c_1\cap c_2$.\\ 
Similarly using Lemma \ref{L3} and Lemma \ref{L4} the $\beta$ dual of $\left(\bar{N}_{p}^{q}\right)$ and $\left(\bar{N}_{p}^{q}\right)_{\infty}$ can be found  same way we can show the other two results as well. 
\end{proof}
 
Let $X\subset\omega$ be a normed space and $a\in\omega$. Then we write
$$\|a\|^*=\sup\left\{\left| \sum_{k=0}^{\infty} a_kx_k\right|:\|x\|=1\right\}$$
provided the term on the right side exists and is finite, which is the case whenever $X$ is a BK space and $a\in X^{\beta}$ \cite[Theorem 7.2.9]{wilansky2000summability}.

\begin{theorem}\label{T3}
For $\left[\left(\bar{N}_{p}^{q}\right)_0\right]^{\beta}$ ,  $\left[\left(\bar{N}_{p}^{q}\right)\right]^{\beta}$ and $\left[\left(\bar{N}_{p}^{q}\right)_\infty\right]^{\beta}$ the norm $\|~.~\|^*$ is defined as

$$\|a\|^*=\sup_{n}\left\{\sum_{k=0}^{n} R_{k}\left|\sum_{j=k}^{n} \left(-1\right)^{j-k}{H_{j-k}^{(p)}\over q_{j}}a_j\right|\right\}$$ 
\end{theorem} 
\begin{proof}
If $x^{[n]}$ denotes the $n$th section of the sequence $x\in \left(\bar{N}_{p}^{q}\right)_0$ then using (\ref{E12}) we have 
\begin{align*}
\tau_{k}^{[n]}=\tau_{k}(x^{[n]})={1\over R_k}\sum_{j=0}^{k}p_{n-j}q_jx_{j}^{[n]}
\end{align*}
Let $a\in \left[\left(\bar{N}_{p}^{q}\right)_0\right]^{\beta}$, then for any non-negative integer $n$ define the sequence $d^{[n]}$ as 
$$d_{k}^{[n]}= \left\{\begin{matrix}
R_{k}\left[\sum_{j=k}^{n} \left(-1\right)^{j-k}{H_{j-k}^{(p)}\over q_{j}}a_j\right] &0\leq k\leq n\\
0& k>n \end{matrix}\right.$$
Let $\|a\|_{\Pi} = \sup_{n}\| d^{[n]}\|_1=\sup_{n}\left(\sum_{k=0}^{\infty}|d_k^{[n]}|\right) $ where $\Pi=\left[\left(\bar{N}_{p}^{q}\right)\right]^{\beta}$. 
Then
\begin{align*}
\left|\sum_{k=0}^{\infty}a_k x_k^{[n]}\right|&=\left|\sum_{k=0}^{n} a_k \left( \sum_{j=0}^{k} \left( -1 \right)^{k-j} 
{H_{k-j}^{(p)}\over q_k} R_j \tau_{j}^{[n]}\right)\right| ~~~~~~~~~~~\text{Using} \eqref{E14}\\
&=\left|\sum_{k=0}^{n}R_k\left(\sum_{j=k}^{n}\left(-1\right)^{j-k} {H_{j-k}^{(p)}\over q_j} a_j\right)\tau_{k}^{[n]}\right| \\
&\leq \sup_{k}|\tau_{k}^{[n]}|\cdot \left(\sum_{k=0}^{n}R_k \left|\sum_{j=k}^{n}\left(-1\right)^{j-k} {H_{j-k}^{(p)}\over q_j} a_j\right|\right) \\
&=\|x^{[n]}\|_{\bar{N}_{p}^{q}}\|d^{[n]}\|_1 \\
&=\|a\|_{\Pi}\|x^{[n]}\|_{\bar{N}_{p}^{q}}.
\end{align*}
Hence \begin{equation} \label{E16}
\|a\|^{*}\leq \|a\|_{\Pi}. 
\end{equation}
To prove the converse define the sequence $x^{(n)}$ for any arbitrary $n$ by 
$$\tau_{k}\left(x^{(n)}\right)=\text{sign}\left( d_{k}^{[n]}\right)\hspace{1cm}(k=0,1,2,\ldots). $$ 
Then 
$$\tau_{k}\left(x^{(n)}\right)=0~~\text{  for } k>n~~ \text{ i.e }~x^{(n)}\in \left(\bar{N}_{p}^{q}\right)_0, \hspace{1cm} \|x^{(n)}\|_{\bar{N}_{p}^{q}}=\|\tau_{k}\left(x^{(n)}\right)\|_{\infty}\leq 1 .$$
and 
\begin{align*}
\left| \sum_{k=0}^{\infty}a_kx_{k}^{(n)}\right|=\left| \sum_{k=0}^{n}d_{k}^{[n]}x_{k}^{(n)}\right|\leq \sum_{k=0}^{n}\left|d_{k}^{[n]}\right|\leq \|a\|^*.
\end{align*}
Since $n$ is arbitrarily choosen so  
\begin{equation} \label{E17} 
\|a\|_{\Pi}\leq \|a\|^{*} .
\end{equation}
From \eqref{E16} and \eqref{E17} we get the required conclusion. 
\end{proof}

\parindent=0mm\vspace{0.01in}
Some well known results that are required for proving the compactness of operators are \\
\begin{proposition}\label{P3}
(cf. \cite{mal98}, Theorem 7) Let $X$ and $Y$ be BK spaces, then $(X,Y)\subset \mathcal{B}(X,Y)$ that is every matrix $A$ from $X$ into $Y$ defines an element $L_A$ of $\mathcal{B}(X,Y)$ where 
$$L_A(x)=A(x)~~~~~~~~~~~~~~\forall~x\in X.$$
Also $A\in (X,\ell_\infty)$ if and only if
$$\|A\|^*=\sup_{n}\|A_n\|^*=\|L_A
\|<\infty .$$ 
If $\left(b^{(k)}\right)_{k=0}^{\infty}$ is a basis of $X , Y$ and 
$Y_1$ are FK spaces with $Y_1$ a closed subspace of $Y$, then 
$A\in (X,Y_1)$ if and only if $A\in (X,Y)$ and $A\left(b^{(k)}\right)
\in Y_1$ for all $k=0,1,2,\ldots$.
\end{proposition}

\begin{proposition}\label{P4} (cf. \cite{mal99}, Proposition 3.4) Let $T$ be a triangle
\begin{enumerate}
\item[(i)] If $X \text{  and  } Y$ are subsets of $\omega$, then $A\in (X,Y_T)$ if and only if $B=TA\in (X,Y)$.
\item[(ii)] If $X \text{  and  } Y$ are BK spaces and $A\in (X,Y_T)$, then 
\begin{align*}
\|L_A\|=\|L_B\|
\end{align*} 
\end{enumerate} 
\end{proposition}
Using Proposition \ref{P3} and Theorem \ref{T3} we conclude the following corollary:
\begin{corollary}\label{C1}
Let $\left(p_k\right)_{k=0}^{\infty},\left(q_k\right)_{k=0}^{\infty}$ be given positive sequences, and $R_n=\sum_{k=0}^{n}p_{n-k} q_k$ then 
\begin{enumerate}
\item[i)] $A\in \left(\left(N_{p}^{q} \right)_\infty, \ell_\infty\right)$  if and only if 
\begin{equation}\label{Co1i}
\sup_{n,m}\left\{\sum_{k=0}^{m} R_{k}\left|\sum_{j=k}^{m} \left(-1\right)^{j-k}{H_{j-k}^{(p)}\over q_{j}}a_{nj}\right|\right\}<\infty
\end{equation}
 and 
\begin{equation}\label{Co1ii}
{A_n H_{n}^{(p)}R\over q} \in c_0~~~\forall~n=0,1,\ldots
\end{equation}
\item[ii)] $A\in \left(\left(\bar{N}_{p}^{q} \right),\ell_\infty\right)$  if and only if condition \eqref{Co1i} holds and 
\begin{equation}\label{Co2}
{A_n H_{n}^{(p)}R\over q} \in c~~~~~~~\forall~n=0,1,2,\ldots
\end{equation}
\item[iii)] $A\in \left(\left(\bar{N}_{p}^{q} \right)_0, \ell_\infty\right)$  if and only if condition \eqref{Co1i} holds.
\item[iv)] $A\in \left(\left(\bar{N}_{p}^{q} \right)_0, c_0\right)$  if and only if condition \eqref{Co1i} holds and
\begin{equation}\label{Co3}
\lim_{n\rightarrow\infty} a_{nk}=0~~~~~\text{ for all  } k=0,1,2\ldots
\end{equation}
\item[v)] $A\in \left(\left(\bar{N}_{p}^{q} \right)_0, c\right)$  if and only if condition \eqref{Co1i} holds and
\begin{equation}\label{Co4}
\lim_{n\rightarrow\infty} a_{nk}=\alpha_k~~~~~\text{ for all  } k=0,1,2\ldots
\end{equation}
\item[vi)] $A\in \left(\left(\bar{N}_{p}^{q} \right), c_0\right)$  if and only if conditions \eqref{Co1i}, \eqref{Co1ii} and \eqref{Co3} holds and
\begin{equation}\label{Co5}
\lim_{n\rightarrow\infty}\sum_{k=0}^{\infty} a_{nk}=0~~~~~\text{ for all  } k=0,1,2\ldots
\end{equation}
\item[vii)] $A\in \left(\left(\bar{N}_{p}^{q} \right), c\right)$  if and only if conditions \eqref{Co1i}, \eqref{Co1ii} and \eqref{Co4} holds and
\begin{equation}\label{Co6}
\lim_{n\rightarrow\infty}\sum_{k=0}^{\infty} a_{nk}=\alpha~~~~~\text{ for all  } k=0,1,2\ldots
\end{equation}
\end{enumerate}
\end{corollary}
From theorem \ref{BKspace},\ref{T3} and Proposition \ref{P4} we conclude the following corollary 
\begin{corollary}
Let $X$ be a BK-space and $\left(p_k\right)_{k=0}^{\infty},\left(q_k\right)_{k=0}^{\infty}$ be positive sequences, $R_n=\sum_{k=0}^{n}p_{n-k} q_k$ then 
\item[i)] $A\in \left(X,\left(\bar{N}_{p}^{q} \right)_\infty\right)$  if and only if 
\begin{equation}\label{Co7}
\sup_{m}\left\|{1\over R_m}\sum_{n=0}^{m}p_{m-n} q_nA_n \right\|^*<\infty
\end{equation}
\item[ii)]  $A\in \left(X,\left(\bar{N}_{p}^{q} \right)_0\right)$  if and only if \eqref{Co7} holds and 
\begin{equation}\label{Co8}
\lim_{m\rightarrow\infty} \left({1\over R_m}\sum_{n=0}^{m}p_{m-n} q_nA_n \left(c^{(k)}\right)\right)=0 ~~~\forall~k=0,1,2\ldots
\end{equation} 
where $\left(c^{(k)}\right)$ is a basis of $X$.
\item[iii)] $A\in \left(X,\left(\bar{N}_{p}^{q} \right)\right)$  if and only if \eqref{Co8} holds and 
\begin{equation}\label{Co9}
\lim_{m\rightarrow\infty} \left({1\over R_m}\sum_{n=0}^{m}p_{m-n} q_nA_n \left(c^{(k)}\right)\right)=\alpha_k ~~~\forall~k=0,1,2\ldots
\end{equation}
\end{corollary}
\section{Hausdorff Measure of Noncompactness}
Let $S$ and $M$ be the subsets of a metric space $(X,d)$ and $\epsilon>0$. Then $S$ is called an $\epsilon-$net of $M$ in $X$ if for every $x\in M$ there exists $s\in S$ such that $d(x,s)<\epsilon$. Further, if the set $S$ is finite, then the $\epsilon-$net $S$ of $M$ is called {\it finite $\epsilon-$net} of $M$. A subset of a metric space is said to be {\it totally bounded} if it has a finite $\epsilon-$net for every $\epsilon>0$ \cite{mur11}.\\
If $\mathcal{M}_X$ denotes the collection of all bounded subsets of metric space $(X,d)$. If $Q\in \mathcal{M}_X$ then the \emph{Hausdorff Measure of Noncompactness} of the set $Q$ is defined by 
$$\chi(Q)=\inf \left\{\epsilon>0: Q \text{  has a finite } \epsilon-\text{net in  }X \right\} $$
 The function $\chi:\mathcal{M}_X\rightarrow [0,\infty)$ is called \emph{Hausdorff Measure of Noncompactness} \cite{Josaf}\\
The basic properties of \emph{Hausdorff Measure of Noncompactness} can be found in (\cite{M1-b}, \cite{malkowsky1}, \cite{Josaf}).\\
Some of those properties are\\
If $Q,Q_1$ and $Q_2$ are bounded subsets of a metric space $(X,d),$ then 
\begin{align*}
\chi(Q)&=0 \Leftrightarrow Q~\text{ is totally bounded set,}\\
\chi(Q)&=\chi(\bar{Q}),\\
Q_1\subset Q_2 &\Rightarrow \chi(Q_1)\leq \chi(Q_2),\\
\chi(Q_1\cup Q_2)&=\max\left\{ \chi(Q_1),\chi(Q_2)\right\},\\
\chi(Q_1\cap Q_2)&=\min\left\{ \chi(Q_1),\chi(Q_2)\right\}.
\end{align*}
Further if $X$ is a normed space then \emph{Hausdorff Measure of Noncompactness} $\chi$ has the following additional properties connected with the linear structure. \\
\begin{align*}
\chi(Q_1+ Q_2)&\leq \chi(Q_1)+\chi(Q_2)\\
\chi(\eta Q)&=|\eta|\chi(Q)~~~~~~~~~~~~~~~~\eta \in \mathbb{C}
\end{align*}
The most effective way of characterizing operators between Banach Spaces is by applying Hausdorff Measure of Noncompactness. If $X$ and $Y$ are Banach spaces, and $L\in \mathcal{B}(X,Y)$, then the Hausdorff Measure of Noncompactness of $L$, denoted by $\|L\|_\chi$ is defined as  
$$\|L\|_\chi=\chi\left(L(S_X)\right)$$
Where $S_X=\{x\in X:\|x\|=1\}$ is the unit ball in $X$. \\
From (\cite{Al-Em03}, Corollary 1.15) we know that 
$$ L ~\text{ is compact if and only if }  ~\|L\|_\chi=0 $$.
\begin{proposition} \label{P11} (\cite{Josaf}, Theorem 6.1.1, $X=c_0$)
Let $Q\in M_{c_0}$ and $P_r:c_0\rightarrow c_0 ~~(r\in \mathbb{N}$ be the operator defined by  $P_r(x)=(x_0,x_1,\ldots, x_r, 0,0,\ldots)$ for all $x=(x_k)\in c_0$. Then, we have 
$$\chi(Q)=\lim_{r\rightarrow \infty}\left(\sup_{x\in Q}\|(I-P_r)(x)\| \right)$$
where $I$ is the identity operator on $c_0$. 
\end{proposition}

\begin{proposition} \label{P12} (cf. \cite{Josaf}, Theorem 6.1.1)
Let $X$ be a Banach space with a Schauder basis $\{e_1,e_2,\ldots \}$, and $Q\in M_{X}$ and $P_n:X\rightarrow X ~~(n\in \mathbb{N}$ be the projector onto the linear span of  $\{e_1,e_2,\ldots , e_n\}$. Then, we have 

\begin{align*}
{1\over a}\lim_{n\rightarrow \infty}\sup &\left(\sup_{x\in Q}\|(I-P_n)(x)\| \right) \leq \chi(Q)\\
&\leq \inf_{n}\left(\sup_{x\in Q}\|(I-P_n)(x)\| \right)\leq 
\lim_{n\rightarrow \infty}\sup \left(\sup_{x\in Q}\|(I-P_n)(x)\| \right)
\end{align*}

where  $a=\lim_{n\rightarrow \infty}\sup \|I-P_n\|$, and $I$ is the identity operator on $c$.\\
If $X=c$ then $a=2$. (see \cite{Josaf}, p.22).
\end{proposition}
   
\section{Compact operators on the spaces $\left(\bar{N}_{p}^{q} \right)_0$, $\left(\bar{N}_{p}^{q} \right)$ and $\left(\bar{N}_{p}^{q} \right)_\infty$}
\begin{theorem}\label{T4}
Consider the matrix $A$ as in Corollary \ref{C1}, and for any integers n,s, $n>s$ set 
\begin{equation}\label{E1}
\|A\|^{(s)}=\sup_{n>s}\sup_{m}\left\{\sum_{k=0}^{m} R_{k}\left|\sum_{j=k}^{m} \left(-1\right)^{j-k}{H_{j-k}^{(p)}\over q_{j}}a_{nj}\right|\right\}
\end{equation} 
If $X$ be either $\left(\bar{N}_{p}^{q} \right)_0$ or $\left(\bar{N}_{p}^{q} \right)$ and $A\in (X,c_0) $. Then 
\begin{equation}\label{E2}
\|L_A\|_{\chi}=\lim_{s \rightarrow\infty} \|A\|^{(s)}.
\end{equation} 
If $X$ be either $\left(\bar{N}_{p}^{q} \right)_0$ or $\left(\bar{N}_{p}^{q} \right)$ and $A\in (X,c) $. Then 
\begin{equation}\label{E3}
{1\over 2}\cdot \lim_{s\rightarrow\infty} \|A\|^{(s)}\leq \|L_A\|_{\chi}\leq \lim_{r\rightarrow\infty} \|A\|^{(s)}.
\end{equation}
and if $X$ be either $\left(\bar{N}_{p}^{q} \right)_0$ , $\left(\bar{N}_{p}^{q} \right)$ or $\left(\bar{N}_{p}^{q} \right)_\infty$ and $A\in (X,\ell_\infty) $. Then 
\begin{equation}
0\leq \|L_A\|_{\chi}\leq \lim_{s\rightarrow\infty} \|A\|^{(s)}.
\end{equation}
\end{theorem} 
\begin{proof}
Let $F=\{x\in X:\|x\|\leq 1\}$ if $A\in (X,c_0) $ and $X$ is one of the spaces $\left(\bar{N}_{p}^{q} \right)_0$ or $\left(\bar{N}_{p}^{q} \right)$, then by Proposition \ref{P11} 
\begin{equation}\label{E4}
\|L_A\|_{\chi}=\chi(AF)=\lim_{s \rightarrow\infty} \left[\sup_{x\in F}\|(I-P_s)Ax\|\right]
\end{equation}
Again using Proposition \ref{P3} and Corollary \ref{C1} we have 
\begin{equation}\label{E5}
\|A\|^{s}=\sup_{x\in F}\|(I-P_s)Ax\|
\end{equation}
From \eqref{E4} and \eqref{E5}  we get 
\begin{align*}
\|L_A\|_{\chi}=\lim_{s \rightarrow\infty} \|A\|^{(s)}.
\end{align*}
Since every sequence $x=(x_k)_{k=0}^{\infty} \in c$ has a unique representation 
$$x=le+\sum_{k=0}^{\infty} (x_k-l)e^{(k)}~~~~~~~~~~~\text{where} ~~l\in \mathbb{C}~~\text{is such that } x-le\in c_0$$
We define $P_s:c\rightarrow c$ by  $P_s(x)=le+\sum_{k=0}^{s} (x_k-l)e^{(k)}$, $s=0,1,2,\ldots$.\\
Then $\|I-P_s\|=2$ and using \eqref{E5}  and Proposition \ref{P12} we get 
\begin{align*}
{1\over 2}\cdot \lim_{s\rightarrow\infty} \|A\|^{(s)}\leq \|L_A\|_{\chi}\leq \lim_{s\rightarrow\infty} \|A\|^{(s)}
\end{align*}
Finally we define $P_s:\ell_\infty\rightarrow \ell_\infty$ by $P_s(x)=(x_0,x_1,\ldots, x_s, 0,0\ldots)$, $x=(x_k)\in \ell_\infty$.\\
Clearly $ AF\subset P_s(AF)+(I-P_s)(AF)$\\
So using the properties of $\chi$ we get 
\begin{align*}
\chi(AF)&\leq  \chi[P_s(AF)]+\chi[(I-P_s)(AF)]\\
&=\chi[(I-P_s)(AF)] \\
&\leq \sup_{x\in F}\|(I-P_s)A(x)\|
\end{align*}
Hence by Proposition \ref{P3} and Corollary \ref{C1} we get\\
\begin{center}
$0\leq \|L_A\|_{\chi}\leq \lim_{s\rightarrow\infty} \|A\|^{(s)}$
\end{center}
\end{proof}
A direct corollary  of the above theorem is
\begin{corollary}
Consider the matrix $A$ as in Corollary \ref{C1}, and $X=\left(\bar{N}_{p}^{q} \right)_0$ or $X=\left(\bar{N}_{p}^{q} \right)$  
then if $A\in (X,c_0)$ or $A\in (X,c)$ we have 
\begin{align*}
L_A \text{  is compact if and only if } \lim_{s\rightarrow \infty}\|A\|^{(s)}=0 
\end{align*}
 Further, for $X=\left(\bar{N}_{p}^{q} \right)_0$ , $X=\left(\bar{N}_{p}^{q} \right)$ or $X=\left(\bar{N}_{p}^{q} \right)_\infty$, if $A\in (X,\ell_\infty)$ then we have   
\begin{align}\label{C111}
L_A \text{  is compact if } \lim_{s\rightarrow \infty}\|A\|^{(s)}=0 
\end{align}
\end{corollary}
In \eqref{C111} it is possible for $L_A$ to be compact although $\lim_{s\rightarrow \infty}\|A\|^{(s)}\not=0$, that is the condition is only sufficient condition for $L_A$ to be compact.\\
For example, let the matrix $A$ be defined as $A_n=e^{(1)}~~~n=0,1,2,\ldots$ and the positive sequences $q_n=3^n~,~n=0,1,2,\ldots$ and $p_0=1, p_1=1, p_k=0~,~\forall~k=2,3,\ldots$ . \\
Then by \eqref{Co1i} we have 
$$\sup_{n,m}\left\{\sum_{k=0}^{m} R_{k}\left|\sum_{j=k}^{m} \left(-1\right)^{j-k}{H_{j-k}^{(p)}\over q_{j}}a_{nj}\right|\right\}=\sup_{m}\left( 2-{2\over3^{m}}\right)=2<\infty$$
Now by Corollary \ref{C1} we know $A\in \left( \left(\bar{N}_{p}^{q} \right)_\infty,\ell_\infty\right)$ . \\
But 
$$\|A\|^{(s)}=\sup_{n>s}\left[ 2-{2\over3^{m}} \right]=2-{1\over 2\cdot 3^{s}}~~~~~~\forall ~s$$
Which gives $\lim_{s\rightarrow \infty}\|A\|^{(s)} =2\not=0$.\\
Since $A(x)=x_1$ for all $x\in \left(\bar{N}_{p}^{q} \right)_\infty$, so $L_A$ is a compact operator. 

\bibliography{References}
\bibliographystyle{mmn}

\end{document}